\documentclass[12pt]{article}

\usepackage{latexsym,amssymb,amsmath}
\newcommand{\C}{\mathbb C}
\newcommand{\z}{\mathbb Z}

\newcommand{\n}{\mathbb N}
\newtheorem{lem}{Lemma}[section]

\newtheorem{ex}[lem]{Example}
\newtheorem{co}[lem]{Corollary}
\newtheorem{thm}[lem]{Theorem}
\newtheorem{prop}[lem]{Proposition}

\newenvironment{proof}{\textbf{Proof.}}{\newline\hspace*{\fill}{$\Box$}\\}

\begin{document}
\title{Virtual finite quotients of finitely generated groups}
\author{J.\,O.\,Button\\
Selwyn College\\
University of Cambridge\\
Cambridge CB3 9DQ\\
U.K.\\
\texttt{jb128@dpmms.cam.ac.uk}}
\date{}
\maketitle
\begin{abstract}
If $G$ is a semidirect product $N\rtimes H$ with $N$ finitely generated
then $G$ has the property that every finite group is a quotient of some
finite index subgroup of $G$ if and only if one of $N$ and $H$ has this
property. This has applications to 3-manifolds and to cyclically 
presented groups, for instance for any fibred hyperbolic 3-manifold
$M$ and any finite simple group $S$, there is a cyclic cover of $M$ whose
fundamental group surjects to $S$. We also give a short proof of the
residual finiteness of ascending HNN extensions of finite rank free
groups when the induced map on homology is injective.  
\end{abstract}
\section{Introduction}

One means of studying a finitely generated group $G$ is to examine the
set ${\cal F}(G)$ consisting of the finite quotients of $G$, as is done
when taking the profinite completion of $G$. Even if $G$ is also a residually
finite group, this might not give us the full picture. For instance it is
unknown whether the following is true (this is problem (F14) in \cite{ny}):
if there is $n\ge 2$ such that the residually finite, finitely generated
group $G$ has ${\cal F}(G)$ consisting of all $n$-generator finite groups, 
then $G$ is isomorphic to the free group $F_n$ of rank $n$.

If a finitely generated group $G$ has many finite quotients then it has
many finite index subgroups too and we can consider ${\cal F}(H)$ for
any $H$ of finite index in $G$. It is the case that ${\cal F}(G)$ and
${\cal F}(H)$ might look rather different, for instance if $G$ is a
perfect group (one which is equal to its commutator subgroup $[G,G]$)
then there will be no non-trivial $p$-groups in ${\cal F}(G)$ but there
could be a complex collection of $p$-groups in ${\cal F}(H)$ for many
primes $p$. This can happen for the fundamental group of a closed hyperbolic
3-manifold.

In this paper we are interested in the question of which finitely generated
groups $G$ have the property that the union of ${\cal F}(H)$ over the
finite index subgroups $H$ of $G$ consists of all finite groups. It is
clear that there are such groups, for instance non abelian free groups
or anything that surjects onto one of these groups. Some other examples
were given in \cite{lr} where it was shown that this property holds for
any finitely generated LERF group (one where every finitely generated 
subgroup is the intersection of finite index subgroups) containing a
non abelian free group. They call our property ``having every finite
group as a virtual quotient''. Moreover various consequences were given
in \cite{lubseg}, which collects together a large number of results on
subgroup growth. In Chapter 3 of this book it is mentioned that our
property, here called ``having every finite group as an upper section'',
holds for finitely generated groups with superexponential subgroup
growth, and also with superpolynomial maximal subgroup growth.

We first find many more examples of finitely generated groups having every
finite group as a virtual quotient by showing in Section 2
that if $S\le G$ then the set of finite quotients ${\cal F}(S)$ is 
contained in the union of ${\cal F}(H)$, where $H$ varies over all finite 
index subgroups of $G$, provided the following holds: whenever $K$ is
normal in $S$ with finite index, there exists a finite index subgroup
$L$ of $G$ with $L\cap S=K$. This is straightforward if we impose that
$L$ is normal in $G$ but our generalisation holds by adapting an
argument in \cite{lr}. As we expect there to be many more finite index
subgroups of $G$ than finite index normal subgroups, we can exploit this
in Section 3 where we examine semidirect products $G=N\rtimes H$, with
$N$ finitely generated. We show that $G$ has every finite
group as a virtual quotient if and only if one of $N$ and $H$ does
(but this need not be true if $N$ is infinitely generated). Thus we can
build up many groups with this property by taking repeated semidirect
products of finitely generated groups
as long as we merely ensure that one of the factors has the
property.

In Section 4 we observe that if $G=N\rtimes H$ for a finitely generated
$N$ which surjects to the finite group $F$ then $N$ is
contained in the finite index subgroup of $G$ which
surjects to $F$ as given by Section 2. We present a simple alternative
proof of this and apply it to semidirect products of the form
$G=N\rtimes\z$. We have that ${\cal F}(N)$ is contained amongst the
union of ${\cal F}(G_n)$, where the $G_n$ are the finite cyclic covers
of $G$. This has applications for closed or finite volume hyperbolic
3-manifolds $M$ which are fibred, such as every 2-generated finite
group (in particular every finite simple group) is a quotient of a
cyclic cover of $\pi_1(M)$. Also the fundamental group of any
virtually fibred hyperbolic 3-manifold 
(it is an open question whether all such
3-manifolds are) has every finite group as a virtual quotient.

We look at cyclically presented groups in Section 5. These are formed
by taking any word $w$ in the free group $F_d$ and obtaining the
group $G_d(w)$ from the
$d$ generator $d$ relator presentation where we take the images of $w$
on cyclically permuting the $d$ generators. These groups have appeared
a lot in the literature and, as we can regard $w\in F_d$ to be in
$F_n$ for $n$ at least $d$, we have an infinite family $G_n(w)$ of
groups on fixing $w$ but varying $n$. We can then ask which group
theoretic properties hold for infinitely many $G_n(w)$ in the family.
By adapting the result in Section 4, we show that given any finite
list of finite simple groups, there exist infinitely many $n$ such that
$G_n(w)$ surjects onto all groups in this list, provided that $w$ comes from
a free-by-cyclic word of rank at least 2. This condition is interpreted
as follows: there is a canonical finite cyclic extension 
$H_n=G_n(w)\rtimes C_n$ for the cyclic group $C_n$ of order $n$ and
$H_n=\langle x,t|r(x,t),t^n\rangle$ with $r\in F_2$ not depending on $n$
and having 0 exponent sum in $t$.
To say $w$ comes from a free by cyclic word means that 
$\langle x,t|r(x,t)\rangle$ has kernel equal to a free group (of finite rank)
of the homomorphism which is the exponent sum of $t$. There are many
free-by-cyclic words and it is very efficient to check whether this
condition holds.

In the last section we consider ascending HNN extensions. The reason for
this is that groups of the form $G=N\rtimes_\alpha\z$ can be formed using an
automorphism $\alpha$ of $N$, whereas an ascending HNN extension
$G=N*_\theta$ generalises this by allowing $\theta$ to be an injective 
endomorphism. We might hope that similar results on virtual finite quotients
hold in this case as well, however we first have to recognise as shown in
\cite{spws} that if $N$ is finitely generated and residually finite
then $G=N*_\theta$ need not be residually finite, in contrast to
semidirect products $G=N\rtimes_\alpha\z$. We adapt their construction
slightly to obtain an example where the only finite quotients of
$G=N*_\theta$ are cyclic. However it was shown in \cite{brsp} using
deep results in algebraic geometry that we do have residual finiteness
when the base $N$ of $N*_\theta$ is a finitely generated free group $F_r$. 
We finish by presenting an elementary proof of this in a special case:
when the map $\theta$ induces an injective homomorphism on the
abelianisation $F_r/[F_r,F_r]$. The proof generalises to an ascending
HNN extension of any finitely generated group $N$ having a prime $p$ such that
$N$ is residually finite $p$ and the homomorphism that $\theta$ induces on the
$p$-abelianisation $N/N^p[N,N]$ is invertible.    

\section{Virtual finite images}

If a group is generated by $n$ elements then any quotient
has this property too, thus no finitely generated group can surject
to all finite groups. However we may ask if every finite group is a
virtual image of a given finitely generated group $G$: this means that for
any finite group $F$ there is a finite index subgroup $H$ of $G$ (for which
we write $H\leq_f G$) with $H$ surjecting to $F$. For instance the free
group $F_n$ of rank $n$ has this property when $n\geq 2$. Other examples
are large groups, where $G$ is large if there is a finite index subgroup
of $G$ which surjects to $F_2$. These include surface groups $\pi_1(S_g)$,
where $S_g$ is the orientable surface of genus $g\ge 2$, and (non abelian)
limit groups. This definition of large
comes from \cite{pr}, where a ``large''
property of groups is defined to be an abstract group property $P$ such that
if $H$ has $P$ and $G$ surjects to $H$ then $G$ has $P$, and if $H\le_f G$ 
then $H$ has $P$ if and only if $G$ has $P$. It is
shown there that if $P$ is a ``large'' property satisfied by one finitely
generated group then any large finitely generated group must have $P$.
\begin{prop}
Having every finite group as a virtual image is a ``large'' property.
\end{prop}
\begin{proof}
We just need to show that if $H\le_f G$ and $G$ has every finite group
as a virtual image then so does $H$.

Suppose that $[G:H]=n$. Given a finite group $F$, we have $L\le_f G$
with a homomorphism $\theta$ from $L$ to
the direct product $F\times\ldots\times F$
of $k$ copies of $F$, where $2^k>n$. Now if $A=L\cap H$ then
$[L:A]\le n$ and moreover $\theta(A)$ also has index at most $n$ in
$F\times\ldots\times F$. Now consider the projection $\pi_i$ from
$F\times\ldots\times F$ to the $i$th factor. If $\pi_i(\theta(A))=F$ for
any $i$ then we are done because $A\le_f H$,
but if not we have that $|\pi_i(\theta(A))|\le |F|/2$.
This means that at most half the elements of $F$ appear in the $i$th
coordinate of the image of $A$ under $\theta$. But if this is true for all
$i$ then $|\theta(A)|\le (|F|/2)^k$ so that $\theta(A)$ has index at least
$2^k$ in $F\times\ldots\times F$, which is a contradiction.
\end{proof}
Thus this implies that
a large finitely generated group has every finite group as a virtual
image, although we noted already that it is easy to see directly.

This property is also considered in \cite{lubseg} Chapter 3
under the title of having
every finite group as an upper section (where a section of $G$ is a quotient
$H/N$ of a subgroup $H$ of $G$, and upper means that $N$, hence $H$, has
finite index in $G$). We will use the two phrases interchangeably throughout.
Theorem 3.1 of this book states that if a finitely
generated group $G$ does not have every finite
group as an upper section then $G$ can have at most exponential
subgroup growth type, whereas free groups and large groups have 
superexponential subgroup growth of type $n^n$. To define these terms,
let $a_n(G)$ be the number of index $n$ subgroups of $G$ and $s_n(G)=
a_1(G)+\ldots +a_n(G)$. If $G$ is finitely generated then $a_n(G)$ is
finite for all $n\in\n$. We say that
$G$ has exponential subgroup growth type if
there exist $a,b>0$ such that $s_n(G)\le e^{an}$ for all large $n$ and
$s_n\ge e^{bn}$ for infinitely many $n$ (and more generally
growth of type $f(n)$ by replacing $e^n$ with $f(n)$).

Thus having subgroup growth type which is bigger than exponential,
meaning that $\mbox{lim sup}(\frac{\mbox{log} s_n(G)}{n})$ is infinite,
is a major restriction as it implies that every finite group is an upper
section. However it is also shown in this book that there exist finitely
generated groups with exponential or slower subgroup growth type which
have every finite group as an upper section. This is done in Chapter 13
Section 2 by considering profinite groups which are the product of various
alternating groups and then taking finitely generated dense subgroups.

We can also count subgroups with a specific property, such as being
maximal, normal or subnormal. Another related result, which is Theorem
3.5 (i) in the same book, states that if $G$ does not have 
every finite group as an upper section then $G$ has polynomial maximal
subgroup growth. Again the converse is not true, with similar examples
verifying this.

Suppose that $H$ is a subgroup of $G$ (with both finitely generated) and
$H$ has every finite group as an upper section. If $H$ has infinite index
in $G$ then one would not expect that this property would transfer across
to $G$, for instance if $G$ is an infinite simple group containing a
non abelian free
subgroup then $G$ has no finite images at all except for the
trivial group $I=\{e\}$.
However there is one obvious situation when this can be
done.
\begin{prop} Suppose that $H\leq G$ and for any finite index normal
subgroup $K$ of $H$ there exists a finite index normal subgroup $N$ of
$G$ such that $H\cap N=K$. If $H$ surjects to the
finite group $F$ then $G$ has $F$ as a virtual image.
\end{prop}
\begin{proof} This is simply because if $K\unlhd_f H$ with $H/K\cong F$
and $H\cap N=K$ for $N\unlhd_f G$ then $NH/N\cong H/K$ with $NH\leq_f G$.
\end{proof}
\begin{co}
Suppose $H\le G$ and $H$ has every finite group as an upper section. If for
all $M\le_f H$ and $K\unlhd_f M$ we have $L\unlhd_f G$ such that 
$L\cap M=K$ then $G$ has every finite group as an upper section.
\end{co}
\begin{proof}
If we have a finite group $F$ and $M\le_f H$ with $K\unlhd_f M$ such that
$M/K\cong F$ then $LM$ also has $F$ as a finite quotient by Proposition 2.2.
\end{proof}

In \cite{lr} Theorem 2.1 states that if $H\leq G$ with $H$ having every
finite group as an upper section and $G$ is LERF, that is every
finitely generated subgroup of $G$ is an intersection of finite index
subgroups, then $G$ also has every finite group as an upper section.
However on examining the proof, it seems that the LERF property is
rather stronger than required for the result to hold and so we offer
a version of this result with essentially the same proof but with
a different hypothesis.

\begin{thm} 
Suppose that $H\leq G$ and for any finite index normal subgroup $K$ of $H$
there exists a finite index subgroup $L$ of $G$, but not necessarily
normal in $G$, such that $H\cap L=K$. If $H$ surjects to the finite group
$F$ then $G$ has $F$ as a virtual image. 
\end{thm}
\begin{proof}
On being given $K\unlhd_f H$ with $H/K\cong F$
and $L\leq_f G$ with
$L\cap H=K$ we let $\Delta$ be the intersection of $hLh^{-1}$ over all
$h\in H$ (so that if $HL$ is a subgroup of $G$, which need not be the
case, then $\Delta$ is the core of $L$ in $HL$). As $\Delta$ is a subgroup,
it is invariant under conjugation by its own elements, and by elements of
$H$ too as this just permutes the terms in the intersection. Thus
$\Delta$ is normal in the subgroup generated by $\Delta$ and $H$, which
therefore is $\Delta H$. We have $\Delta H/\Delta\cong H/(\Delta\cap H)$
and we now show that $\Delta\cap H=K$: certainly $\Delta\leq L$ so
$\Delta\cap H\leq L\cap H$. Conversely we need to show that for any
$h\in H$ and $k\in K$ we have $h^{-1}kh\in L$, but this is true because
$K$ is normal in $H$ and is contained in $L$. Finally $\Delta\leq_f G$
because $L\leq_f G$ implies that $L\cap H\leq_f H$, with the elements
of $L\cap H$ conjugating $L$ to itself.
\end{proof}
\begin{co}
Suppose $H\le G$ and $H$ has every finite group as an upper section. If for
all $M\le_f H$ and $K\unlhd_f M$ we have $L\le_f G$ such that
$L\cap M=K$ then $G$ has every finite group as an upper section.
\end{co}
\begin{proof}
This is the same proof as Corollary 2.3 but applying Theorem 2.4 instead of
Proposition 2.2.
\end{proof}

The Long-Reid result in \cite{lr} follows because if $G$ is LERF and there
is a subgroup $M$ of $G$
with $K\unlhd_f M$ such that $M/K$ is the finite group
$F$ then we can find $L\le_f G$ with $L\cap M=K$. This is done by taking
coset representatives $e=m_1,m_2,\ldots ,m_n$ of $K$ in $M$ and finding
$L_i\le_f G$ for $2\le i\le n$ with $K\le L_i$ but $m_i\notin L_i$,
then intersecting the $L_i$.

\section{Semidirect Products}

In order to apply Corollary 2.5, one needs a wide class of groups where
not all members are (or are known to be) large or LERF. Semidirect products
provide such a class. If we have groups $N$ and $H$ with an homomorphism
$\theta:H\rightarrow \mbox{Aut}(N)$ then we can form the semidirect product
$G=N\rtimes_\theta H$ with $G/N\cong H$. Given a group $G$ we may be
able to see it as an internal semidirect product, by finding subgroups
$N$ and $H$ of $G$, with $N$ normal, where $NH=G$ and $N\cap H=I$. In this
case the homomorphism $\theta$ is given by the conjugation action of $H$
on $N$.

One nice feature of semidirect products $G=NH$ is that their subgroup structure
is not too complicated. If $L$ is a subgroup of $G$ then it might not be the
case that $L=SR$ for $S\le N$ and $R\le H$: indeed this is not even true
for direct products. 
However finite index subgroups of a semidirect product are 
``not too far away'' from having this structure.

\begin{prop} Suppose that $G=N\rtimes_\theta H$. Then
any finite index subgroup $L\le_f G$ contains a finite index subgroup of
the form $S\rtimes_\theta R$, where $S\le_f N$ and $R\le_f H$.
Moreover if $N$ is finitely generated and
we are given any finite index 
subgroup $S\le_f N$ then we can find $L\le_f G$ with $L\cap N=S$.
\end{prop}
\begin{proof}
Given $L\le_f G$ we have that $S=L\cap N$ has finite index in $N$
and is also normal in $L$. On setting  $R=L\cap H$ 
we have that $S$ is preserved under conjugation by $R$ so $SR$ is the
subgroup generated by $S$ and $R$ with $S\unlhd SR$, thus $SR$ is also
a semidirect product. Also $SR$ has finite index in $G$: for this we
can assume that $L\unlhd G$ by replacing $L$ with a smaller finite index
subgroup which will only reduce $SR$.
Then on taking left coset representatives $n_i$ for $S$ in $N$ and
$h_j$ for $R$ in $H$, we have that any $g\in G$ is equal to $nh$ for
$n\in N$ and $h\in H$, thus also equal to $n_ish_jr$ for $s\in S, r\in R$.
But this is equal to $n_ih_js'r\in n_ih_jSR$ because here $S\unlhd G$.

Now suppose we have $S$ with $[N:S]=i$ and note that the set 
${\cal S}_i$ of index $i$ subgroups 
of $N$ is finite because $N$ is finitely generated. As $H$
acts on ${\cal S}_i$ by conjugation, the stabiliser $R$ of $S$ in $H$ has
finite index in $H$ and we can form the finite index subgroup
$L=S\rtimes_\theta R$ of $G$ where we restrict $\theta$ from $R$ to 
$\mbox{Aut}(S)$. But if $g\in SR\cap N$, so that we have respective elements
$s,r,n$ with $g=sr=n$ then $r\in N\cap R=I$, thus $g\in S$.
\end{proof}

If $G=N\rtimes H$ and $G$ is finitely generated then so is $H$ as it is
a quotient of $G$. However this need not imply that $N$ is finitely 
generated, so our main interest will be in semidirect products where 
$H$ and $N$ (hence $G$) are finitely
generated. If $H$ has every finite group as a virtual quotient then
so does $G$ (indeed this applies if $H$ is merely a quotient of $G$, by the
correspondence theorem). However we can now prove the more
surprising fact that the same is true with $N$ and $H$ swapped.
\begin{co} If $G=N\rtimes H$ with $N$ finitely generated,
and $N$ has every finite group as a virtual quotient then so does $G$.
\end{co}
\begin{proof}
We need to show that the conditions of Corollary 2.5 are satisfied, where
in the hypothesis $H$ has now become $N$. Given
$M\le_f N$ and $K\unlhd_f M$, we can find $R\le_f H$ such that $MR$ is also
a semidirect product by the second part of the proof of Proposition 3.1.
Now on applying this proposition again to $MR=M\rtimes R$, 
we obtain $L\le_fMR\le_f G$ with $L\cap M=K$
because $M$ is finitely generated too.  
\end{proof}

Notes: (1) We certainly need $N$ to be finitely generated in Corollary
3.2 as the example in \cite{bminf} is a semidirect product
$G=F_\infty\rtimes\z$ where $F_\infty$ is a free group of infinite
rank but the only finite quotients of $G$,
and of its finite index subgroups, are cyclic.\\
(2) Although the conditions of Corollary 2.5 are satisfied for semidirect
products $N\rtimes H$ when $N$ is finitely generated,
we remark that the conditions in Corollary 2.3 need not be.
For instance, take $H=F_n$ with $G=F_n\rtimes_\theta\z$, $M=H$ and $K$ an
index 2 subgroup of $H$. If there is $L$ normal in $G$ with $L\cap H=K$
then $L\cap H$ is the intersection of normal subgroups and so is normal
in $G$ too. This means that $tKt^{-1}=K$ where $t\in G$ generates the
factor $\z$, thus forcing $\theta(K)=K$ which need not be the case.

We now say a few words on largeness and LERF of semidirect products as
in the statement of Corollary 3.2, so that $G=N\rtimes H$ with both
factors finitely generated and $N$ has every finite group as a virtual
quotient. First if we have a direct product
$G=N\times H$, or if $H$ is finite so that $N\unlhd_f G$, then the
corollary says nothing new. For a direct product $G$, if one factor is large
or both are LERF then $G$ is large or LERF respectively. Moreover if $H$
is finite then $G$ is large or LERF if and only if $N$ is large or LERF
respectively.

Now suppose that $H$ is not normal in $G$ (which is equivalent
to $G=N\rtimes H$ not being the direct product of $N$ and $H$) but
is infinite. Let us take $H$ to be the smallest infinite group $\z$ and
$N$ to be a group known to have all finite groups as virtual images. 
If $N$ is the
free group $F_n$ for $n\ge 2$ then $F_n\rtimes\z$ need not be LERF in
general by \cite{burks}, and it is known to be large if it contains
$\z\times\z$ by \cite{bis} but otherwise this is open. 
Similarly if $N=\pi_1(S_g)$
for $S_g$ the closed orientable surface of genus $g\ge 2$ then
$\pi_1(S_g)\rtimes\z$ need not be LERF in general (for instance one can
put together two copies of the above example for $F_n\rtimes\z$ so that
the resulting group contains subgroups which are not LERF), and
although some groups of this form have been proved large by geometric
considerations, the question of whether all of these groups are large is 
very much open
too. Thus Corollary 3.2 tells us that all groups of the form
$F_n\rtimes\z$ or $\pi_1(S_g)\rtimes\z$ have every finite group as a
virtual quotient.   

We now look at the reverse situation where $G=N\rtimes H$ and $G$ has every 
finite group as an upper section, to see what this implies for $N$ or $H$.

\begin{lem} A group $G$ has every finite group as an upper section if and
only if it has infinitely many distinct alternating groups $A_n$ as upper
sections.
\end{lem}
\begin{proof}
Every finite group $F$ is a subgroup of $A_N$ for some $N$ (and hence for
all $n\ge N$): this is clear for $S_N$ and if the resulting subgroup has
odd permutations then we can increase $N$ by 2 and add a 2-cycle to the
odd elements.

Now for $F$ and $N$ above, suppose that we have $L\le_f G$ with a 
surjection $\theta$ from $L$ to $A_n$ for some $n\ge N$. As $F\le A_n$
we can pull it back to get $\theta^{-1}(F)\le_f L\le_f G$ with
$\theta\theta^{-1}(F)=F$.
\end{proof}

\begin{thm} If $G=N\rtimes H$ and $G$ has every finite group as an upper 
section then either $N$ or $H$ does too.
\end{thm}
\begin{proof}
We can assume that there is $n_0\ge 5$ such that for all $n\ge n_0$ the
group $A_n$ is not an upper section of $H$, as otherwise we are done by
Lemma 3.3. Now given any $n$ which is at least $n_0$, we know there is 
$L\le_f G$ and a surjection $\theta:L\rightarrow A_n$. As $N\unlhd G$ we
have that $S=L\cap N\unlhd L$. Thus $\theta(S)$ is normal in $A_n$, meaning
that if we can eliminate $\theta(S)=I$ we obtain $\theta(S)=A_n$, and as
$L\le_f G$ we get $S\le_f N$ so $A_n$ is an upper section of $N$ and we
are done by Lemma 3.3 again.

Now if $\theta(L)=A_n$ but $\theta(S)=I$ we see that $\theta$ factors
through $L/S\cong LN/N$. But $LN/N$ is a finite index subgroup of
$G/N\cong H$, which does not have $A_n$ as an upper section.
\end{proof}

We can form repeated semidirect products $G=G_1\rtimes G_2\rtimes\ldots\rtimes
G_n$ for finitely generated groups $G_i$, where for this to be defined we
would need to bracket all terms in some way and provide the appropriate
homomorphisms. What Corollary 3.2 and Theorem 3.4 show is that no matter
how this is done, $G$ has every finite group as an upper section if and
only if at least one of the $G_i$ does.  Consequently it could be
argued that the semidirect product is behaving like a direct product
for this property.

We also have the following.
\begin{co}
If $G$ is a repeated semidirect product of finitely generated groups
$G_1,\ldots ,G_n$ and $G$ has bigger than exponential subgroup growth
type then at least one of the $G_i$ has every finite group as an upper
section. Conversely if one of the $G_i$ has 
bigger than exponential subgroup growth
type then $G$ has every finite group as an upper section. 
\end{co}
\begin{proof} The subgroup growth condition implies that $G$ or $G_i$
has every finite group as an upper section, so now repeatedly
apply Theorem 3.4 or Corollary 3.2 respectively.
\end{proof}

We remark that if $G=N\rtimes H$ and $N$ has bigger than exponential
subgroup growth then it is not known whether $G$ does. For instance if
$N=F_n$, which has subgroup growth of superexponential type, and
$H=\z$ then we have subgroup growth of superexponential type for $G$ if
it contains $\z\times\z$ because of largeness, but even exponential
growth is not known in the other cases. At least we see here that $G$ has
the weaker property of having every finite group as an upper section.

\section{Cyclic covers of groups and fibred manifolds}

If we look back to Theorem 2.4 and assume the conditions are
satisfied, where we now replace $H$ with $S$,
we conclude that if $S\le G$ and $S$ surjects to the finite
group $F$ then a finite index subgroup of $G$ surjects to $F$ as well.
But if we examine the proof, we see that this subgroup contains $S$. If
we now specialise to the case where $G=N\rtimes H$ for $N$ finitely
generated with $N=S$ above then any subgroup $L$ of $G$ which contains
$N$ must be of the form $L=N\rtimes (H\cap L)$, because if $x\in H\cap L$
then $nx\in L$ for all $n\in N$.

We give here a quick alternative proof for semidirect products
which can be more useful for constructive purposes.
\begin{prop}
If $G=N\rtimes H$ for $N$ and $H$ finitely generated and $N$ surjects to the
finite group $F$ then there exists $L\le_f G$ containing $N$ with $L$
surjecting to $F$.
\end{prop}
\begin{proof}
If $K\unlhd_f N$ with $N/K\cong F$ then, as $N$ is finitely generated,
we have a finite index characteristic subgroup $C$ of $N$ which is
contained in $K$. Consequently $N/C$ surjects to $F$ too and the
conjugation action of $H$ on $N$ descends to one on $N/C$. As $N/C$ is
finite, we take the intersection $S\le_f H$ of all point stabilisers of
this action, so $sns^{-1}=nC$
for all $n\in N$ and $s\in S$. We then let $L=N\rtimes S$ and observe that
the surjection from $N$ to $F$ via $N/C$ extends to $L$ by sending all of
$S$ to the identity.
\end{proof}

This result might not be of interest if there is no reason to favour
finite index subgroups of $G$ that contain $N$ over other finite
index subgroups. However there is one setting, motivated by topology,
where these subgroups are given an important r\^{o}le. This is when
$G=N\rtimes_\alpha\z$ for $\alpha$ an automorphism of $N$. If $\z$ is
generated by the element $t$ then we define the {\bf cyclic cover} $G_n$
of $G$ to be the index $n$ subgroup $\langle N,t^n\rangle$. We have that
$G_n\unlhd G$; indeed we can think of $G_n$ as being the kernel of
the map from $G$ to the cyclic group $C_n$ given by the exponent sum of
$t$ modulo $n$. If $N$ has
a presentation $\langle g_1,\ldots ,g_k|r_1,r_2,\ldots\rangle$ then a
presentation for $G$ would be
\[\langle g_1,\ldots ,g_k,t|r_1,r_2,\ldots ,tg_1t^{-1}=\alpha(g_1),
\ldots ,tg_kt^{-1}=\alpha(g_k)\rangle\]
and for $G_n$ we have
\[\langle g_1,\ldots ,g_k,s|r_1,r_2,\ldots ,sg_1s^{-1}=\alpha^n(g_1),
\ldots ,sg_ks^{-1}=\alpha^n(g_k)\rangle\]
where $s=t^n$.

In particular all of the $G_n$ are generated by $k+1$ elements.
The connection with topology is that if $N=\pi_1(M)$, the
fundamental group of a $d$ dimensional manifold $M$ then on taking a
homeomorphism $h$ of $M$, we can form the $d+1$ dimensional manifold which
is fibred over the circle $S^1$ with fibre $M$ using $h$, and this has
fundamental group $N\rtimes_{h_*}\z$ where $h_*$ is the automorphism
of $N$ induced by $h$. If $d=2$ then $N$ must be the fundamental group
of a surface, and if this surface is compact and orientable 
then $N=F_n$ for a bounded surface and $\pi_1(S_g)$ if it is closed.
In particular this discussion applies to fibred knots, where $N$ is free
and the cyclic covers take on particular importance. We see that in the
case where $G=N\rtimes\z$, any finite index subgroup $L$ of $G$ that
contains $N$ is a cyclic cover, as if $G=\langle N,t\rangle$ then
$L=\langle N,t^n\rangle$ for $n$ the smallest positive integer with
$t^n\in L$. Thus if there is a surjection from $N$ to a finite group $F$
then there is 
also a surjection from a cyclic cover of $G$ to $F$. Consequently we
see all the finite images of $N$ amongst the finite images of the cyclic
covers of $G$. We also have:
\begin{co}
If the finitely generated group $N$ surjects to the finite group $F$
then for any $G=N\rtimes_\alpha\z$ there are infinitely many cyclic
covers of $G$ that surject to $F$.
\end{co}
\begin{proof}
On looking at the proof of Proposition 4.1 we see that the cyclic cover
$G_n$ of $G$ surjects to $F$ provided that the automorphism of $N/C$
induced by $\alpha$ satisfies $\alpha^n=\mbox{id}$ (or instead of $C$ any
normal subgroup $J$ of $N$ which is contained in $K$ and fixed by
$\alpha$ would do).
Thus any integer multiple of $n$ works too.
\end{proof}

This observation has various consequences if we are interested in 
(non abelian) finite simple images of groups, as by the classification
of finite simple groups all of them are 2-generated.
\begin{co}
If the finitely generated group $N$ surjects to the free group $F_2$
and $G$ is any group of the form $N\rtimes_\alpha\z$ then for any
finite list of finite simple groups $S_1,\ldots ,S_l$ there exists
infinitely many cyclic covers of $G$ which surject to all of $S_1,\ldots
,S_l$.
\end{co}
\begin{proof}
As $N$ surjects to $F_2$ it consequently surjects to each of the $S_i$.
Applying Corollary 4.2 gives us an integer $n_i$ such that any cyclic
cover of $G$ having index 0 modulo $n_i$ surjects to $S_i$. Thus we can
take any multiple of the lowest common multiple of $n_1,\ldots ,n_l$.
\end{proof}

Again applications are provided by compact 3-manifolds which are
fibred over the circle, as if the fibre is a surface of negative
Euler characteristic then the fundamental group is either (non abelian)
free, a closed orientable surface group (of genus at least 2), or
a closed non-orientable surface group (of genus $g\ge 3$) with
fundamental group having a presentation
$\langle x_1,\ldots ,x_{g}|x_1^2x_2^2\ldots x_{g}^2\rangle$. This also
surjects to $F_2$ unless $g=3$ (see \cite{LS} page 52). In particular if $M$ is
an orientable hyperbolic 3-manifold (with or without boundary) that 
fibres over the circle then any finite simple group is a quotient of
a cyclic cover of $M$.

We now say a few words on what is known about the finite simple images of
the fundamental group of a hyperbolic 3-manifold $M$ of finite volume. 
It is certainly true that    
$M$ has infinitely many images of type $PSL(2,\mathbb F)$ for $\mathbb F$ a
finite field, coming from the fact that $\pi_1(M)$ is a subgroup of
$PSL(2,\C)$. However a lot less seems to be known about other types.
There exist examples $M$, both closed and with boundary, where
$\pi_1(M)$ surjects to $F_2$ and so all finite simple groups appear. 
A range of results are obtained in \cite{dt} which fixes a
finite simple group $F$ and considers the question of whether $\pi_1(M)$
has $F$ as a quotient from a probabilistic point of view. In particular
the authors take a genus $g\ge 2$ and define the concept of a manifold $M$ of a
random Heegaard splitting of genus $g$ and complexity $L$. There are
only finitely many of these for fixed $g$ and $L$ so the probability
that $M$ has a particular property, such as $\pi_1(M)$ surjecting to
a finite group $F$, 
is well defined. It is shown in Proposition 6.1 that for any $F$
this probability tends to a limit $p(F,g)$ as $L$ tends to infinity.
Moreover when $F$ is any non abelian finite
simple group, Theorem 7.1 obtains the limit of $p(F,g)$ as $g$ tends to
infinity: in particular it is strictly between 0 and 1.

But if we specialise to a family of simple groups not involving
$PSL(2,\mathbb F)$ then things are less clear: for instance
Question 7.6 of this paper asks whether every closed (or finite
volume) hyperbolic
3-manifold has a quotient $A_n$ for some $n\ge 5$. Their
Theorem 7.7 does state that given $\epsilon>0$ there is $g_0>0$
such that for all $g\ge g_0$ the probability of a manifold $M$ of a
random Heegaard splitting of genus $g$ having a quotient $A_n$
for some $n\ge 5$,
or even at least $k$ quotients of distinct groups $A_n$ for fixed $k$,
is at least $1-\epsilon$. We can have variants on this question: does
every closed (or finite volume) hyperbolic 3-manifold have a quotient
$A_n$ for infinitely many $n$ or all but finitely many $n$? We do not
know of a specific example proven not to have either one of these
properties.

We did however locate an example of a closed hyperbolic 3-manifold which
surjects to all but finitely many, but not all $A_n$. In \cite{cmt} the
extended $[3,5,3]$ Coxeter group $\Gamma$, which is now known to be the
fundamental group of the smallest closed non orientable hyperbolic 3-manifold,
is studied along with its orientable double cover $\Gamma^+$ (the smallest
orientable example). Theorem 4.1 in this paper states that for all large
$n$, $A_n$ and $S_n$ are quotients both of $\Gamma$ and of $\Gamma^+$.
This is proved by an intricate argument that links together copies 
of particular permutation representations of each group in order to form
transitive permutation representations of arbitrarily large degree.
However it is easily checked, using the given presentation and MAGMA or
GAP, that $\Gamma^+$ does not surject to small $A_n$. If one wants a
3-manifold, rather than a 3-orbifold, with this property then they show that
the group $\Sigma_{60a}$ with index 60 in $\Gamma^+$ is torsion free,
thus $\mathbb H^3/\Sigma_{60a}$ is a closed orientable hyperbolic 3-manifold.
Again the computer tells us that it does not surject to $A_n$ for
$n=5,6,7,10$ (though it does for 8 and 9). As for large $n$, any
homomorphism sending $\Gamma^+$ to $A_n$ maps $\Sigma_{60a}$ to a subgroup
of index at most 60, which must be $A_n$ for $n\ge 61$. In particular
this observation shows that if a group $G$ surjects to infinitely many
$A_n$, or all but finitely many $A_n$, then any finite index subgroup
has this property too.

Another point of interest in this question is provided by \cite{lubseg}
Theorem 3.5 (iii)
which states that if a finitely generated group surjects to only
finitely many groups from all $A_n$ and $S_n$ then the growth type
for the number of maximal subgroups of index $n$ is at most $n^{\sqrt n}$.
Now for the free group $F_2$ it is $n^n$ (the same growth type as for
all subgroups of index $n$) and, as maximal subgroups pull back to
maximal subgroups under any surjection by the correspondence theorem,
any hyperbolic 3-manifold with a surjection from its fundamental group
to $F_2$ must also have this property. Thus if there exist hyperbolic
3-manifolds with only finitely many surjections to $A_n$ and $S_n$, we
would witness two markedly different types of growth of a natural
quantity purely within the class of hyperbolic 3-manifolds.

We finish this section by pointing out that if one returns to the
property of having every finite group as an upper section then any
closed or finite volume hyperbolic 3-manifold which is virtually fibred,
that is having a finite cover or equivalently a finite index subgroup
which is fibred, has this property by dropping down and using
Corollary 3.2. A famous question of Thurston asks whether this is this
case for all closed or finite volume hyperbolic 3-manifolds. Having every
finite group as an upper section would also follow from \cite{lr} Theorem
2.1 if every such 3-manifold had a fundamental group which is LERF.
Currently both these questions are open. Another notorious open
question concerns word hyperbolic groups and whether they are always
residually finite. It is shown in \cite{olsmp} that if this is true
then every (non elementary) word hyperbolic group would have infinitely
many non abelian finite simple quotients. It does not however imply 
unconditionally that a given non elementary word hyperbolic group $H$ has
non abelian finite simple quotients. If one could find such 
a group $H$ and prove that
it has no (or just finitely many) non abelian finite simple quotients
then the above result implies that $H$  has some quotient which is also
non elementary word hyperbolic but not residually finite. However let us
quote a comment from \cite{dt} as a warning
for when a computer program has been given a finite
presentation but not found any non abelian finite simple quotients. This
quote relates to random presentations of deficiency zero, but it could
well apply in other situations:\\
``In any case, for a typical deficiency 0 group that has no quotients among
the first few non-abelian simple groups, it is clear that if it has any such
quotient, the index must be so astronomically large as to be far beyond 
brute force computation.''   

\section{Finite images of cyclically presented groups}

The literature on cyclically presented groups is extensive, so rather
than a full list of papers and summary of results we content ourselves
with citing \cite{ce} as a very recent paper on the subject: any interested
reader could then follow back the references therein. We also outline
problems of this area that are most relevant to us here.

Let $v$ be an element (without loss of generality cyclically reduced)
of the free group $F_n$ with free generating set
$x_0,\ldots ,x_{n-1}$. There is a natural automorphism $\pi_n$
of $F_n$ having order $n$, obtained from the $n$-cycle $(1,2,\ldots ,n)$.
We define the {\bf cyclically presented group} $G_n(v)$ to be the
group obtained from the deficiency zero presentation
\[\langle x_0,\ldots ,x_{n-1}|v(x_0,\ldots,x_{n-1}),\pi_n
(v(x_0,\ldots,x_{n-1})),
\ldots ,\pi_n^{n-1}(v(x_0,\ldots ,x_{n-1}))\rangle;\]
of course $\pi_n^i(v(x_0,\ldots,x_{n-1}))=v(x_i,\ldots,x_{n-1+i})$ where we
always take subscripts modulo $n$. In particular the action of $\pi_n$
descends to $G_n(v)$.
One then asks, given $n\in\n$ and
any such word $v\in F_n$,
what properties does $G_n(v)$ have: for instance is it
non-trivial, infinite, non abelian or moreover non soluble? Does it contain
$F_2$ or is it even a large group? Some approaches to this question have
looked at a range of $v$ and where $n$ is less than some bound. However
as a (cyclically reduced) word of $F_d$ is also one in $F_n$ for $n\ge d$, 
our focus will be to fix $v\in F_d$ and then examine this
infinite family of groups over all $n\ge d$. We can then ask: does
$G_n(v)$ have any of the above properties for infinitely many $n$? How
about for all but finitely many $n$?

Before listing what is already known about the above problem, 
we mention an approach to these questions which is much used and which
will be especially suitable for us here. On being given a cyclically
reduced $v\in F_d$ and taking letters $x$ and $t$,
we can define the word $w(x,t)=v(x,txt^{-1},\ldots ,$ 
$t^{d-1}xt^{-(d-1)})\in F_2$ (which need not be cyclically reduced but it
does not matter).
If we define the group $H_n(w)$ by the 2 generator 2 relator
presentation
$\langle x,t|w(x,t),t^n\rangle$ for any $n\in\n$ then it is easily verified 
that $H_d(w)=G_d(v)\rtimes_{\pi_d} C_d$ for $C_d$ the cyclic group of order
$d$, and this is also true on replacing $d$ throughout with $n$ when $n\ge d$.
Thus on being given $v$, instead of considering 
the family $G_n(v)$ we can convert $v$ into $w$ and examine the
corresponding family $H_n(w)$. We have a straightforward lemma which will
be of use.
\begin{lem} (i) Given a cyclically reduced element $v\in F_d$ and the
corresponding families of groups $G_n(v)$ and $H_n(w)$ for $n\ge d$,
we have that $G_n(v)$ is non-trivial, infinite, non soluble,
contains $F_2$ or is large if and only if $H_n(w)$ is not cyclic of
order $n$, infinite, non soluble,
contains $F_2$ or is large respectively.\\
(ii) If there exists $n$ such that $G_n(v)$, respectively $H_n(w)$,
has one of the corresponding properties above then this also holds for
$G_m(v)$, respectively $H_m(w)$, whenever $n$ divides $m$.
\end{lem}
\begin{proof}
(i) Any property which is preserved by finite index subgroups and extensions
by finite cyclic groups will apply equally to $G_n(v)$ and $H_n(w)$. As for
triviality, if $H_n(w)=C_n$ then we cannot have $|G_n(v)|>1$, nor can it
be infinite.\\
(ii) The point is that if $n$ divides $m$ then $H_m(w)$ surjects to
$H_n(w)$ by adding the relation $t^n$.
By abelianising the presentation for any $H_n(w)$, we see that
it must surject to $C_n$ with $t$ mapping to a generator because the
exponent sum of $t$ in $w(x,t)$ is zero. Now if $H_m(w)=C_m$ then it
is the cyclic group $\langle t|t^m\rangle$, thus $H_n(w)$ is the
group $\langle t|t^m,t^n\rangle=C_n$. 
All other properties mentioned above for $H_n(w)$ are preserved
by prequotients, and we can now use Part (i) to transfer these over to 
$G_n(v)$.
\end{proof}
Thus we have that possessing any of these properties for some $n$ in a
family of the form $\{G_n(v):n\ge d\}$ or $\{H_n(w):n\ge d\}$ 
implies the same for infinitely many $n$.

Given such a $v\in F_d$ (and corresponding $w(x,t)\in F_2$) we can form
the {\bf associated polynomial} $f_v(t)\in\z[t]$ of degree at most
$d-1$ which is
defined by $f_v(t)=a_{d-1}t^{d-1}+\ldots +a_1t+a_0$, where $a_i$ is the
exponent sum of the letter $x_i$ in $v$ (in fact this is just the
Alexander polynomial of $w(x,t)$, although the latter is a Laurent
polynomial which is only defined
up to multiplication by $\pm t^j$ for any $j\in\z$). This polynomial
plays an important part in the theory of cyclically presented groups,
although of course different words $v_1,v_2\in F_d$ 
can have the same polynomial
$f_{v_1}=f_{v_2}$; indeed this happens if and only if $v_1$ and $v_2$
represent the same element in the abelianisation $\z^d$ of $F_d$.

We now give a brief outline of known results in relation to the problems
above. First we ask: is there $v$ such that $G_n(v)$ is trivial for infinitely
many $n\ge d$? The only examples we know are the obvious ones $v=x_i^{\pm 1}$
for $0\le i\le d-1$, giving $w=t^ixt^{-i}$.
In \cite{ce} the abelianisation $G_n(v)^{\mbox{ab}}
=G_n(v)/[G_n(v),G_n(v)]$
of $G_n(v)$ is studied, using results in \cite{c}. Let us say that $f_v$
is of cyclotomic type if it is a product of (not necessarily distinct)
cyclotomic polynomials $\pm\Phi_m$, up to multiplication by $\pm t^i$.
It is shown there that if $f_v$ is not of cyclotomic type (and $f_v\ne\pm t^i$)
then $G_n(v)^{\mbox{ab}}$
is non trivial for all but finitely many $n$ and if $f$ is of
cyclotomic type then
$G_n(v)^{\mbox{ab}}$ is non trivial for infinitely many $n$. Thus only
$v$ with an associated polynomial of cyclotomic type could possibly yield
$G_n(v)$ being trivial for infinitely many $n$ and only $\pm t^i$ could do
for all $n$, but we do not know of any examples other than the above
in either case. We remark
in Example 5.4 that there exist cases where $f_v(t)=\pm t^i$ but which
need other methods to show non triviality.

However any hope that there might be equivalent statements or open questions
for the other properties can be dispelled because there exist examples
where $G_n(v)$ is finite cyclic for all $n$. For instance Theorem 3 of
\cite{prf} implies that if we take $v=x_0^\alpha x_1^{-\beta}$ with $\alpha$
and $\beta$ integers at least 2 which are coprime then $G_n(v)$ is cyclic
of order $\alpha^n-\beta^n$. But if we are happy to argue ``generically''
then we can use a powerful theorem of Ol'shanski\u{i} in \cite{Olhp}. This
states that if $H$ is a non elementary word hyperbolic group and $g\in H$
is any element of infinite order then the quotient $Q_n(g)$ of $H$ formed by
adding the relator $g^n$ is again non elementary word hyperbolic for all
but finitely many $n$. In particular we see that if the 2 generator
1 relator group $H(w)=\langle x,t|w(x,t)\rangle$ is a 
non elementary word hyperbolic group then $H_n(w)$
will be non elementary word hyperbolic as well for all but finitely many $n$,
thus these groups will certainly be infinite, non soluble and will
contain $F_2$. However we note that we can certainly have $H(w)$ not
word hyperbolic but $H_n(w)$ is for all large $n$, as when performing
Dehn filling which is mentioned later.
The claim of genericity holds because given a free group $F_r$ of finite rank
$r\ge 2$ and any word in $F_r$ of length at most $l$, the proportion of
1-relator groups thus obtained which are non elementary word hyperbolic
tends to 1 as $l$ tends to infinity. (Here the situation is slightly
different as our word $w(x,t)$ has exponent sum zero in $t$, but we
can always apply a change of basis to put an arbitrary word in this form.)    

There is also a parallel statement for largeness: if $H$ is a finitely
generated large group and $g\in H$ is any element 
then Lackenby shows in \cite{lac} that the quotient $Q_n(g)$ formed by
adding the relator $g^n$ is also large, but now only for infinitely
many $n$. Note that if we get largeness of $Q_n(g)$ for one value of $n$ then 
we immediately get largeness for all $Q_{kn}(g)$ with $k\ge 1$ as $Q_{kn}(g)$
surjects to $Q_n(g)$. There is a short proof of this result in \cite{olos},
but Corollary 3.5 of \cite{belos} gives an example of a large torsion free
hyperbolic group $H$ and an element $g\in H$ such that $Q_n(g)$ is not large
for odd $n$. Consequently we see that by starting with a large group of the
form $H(w)=\langle x,t|w(x,t)\rangle$ with the exponent sum of $t$ being
zero, for which there are lots of examples, then we obtain families
$H_n(w)$ and $G_n(v)$ which are large for infinitely many $n$, but we do not
know of examples which are large for all but finitely many $n$.

However let us now concentrate on the finite quotients of $H_n(v)$ and
$G_n(w)$. We might ask under what circumstances do we find examples having
non abelian finite simple quotients: this need not be implied by the fact
that the group is infinite or non soluble, and is not known to be implied
if the group is non elementary word hyperbolic or large. However there is
a class of groups where we can use the results in Section 4 to get a
positive answer. Suppose that $v$ is any cyclically reduced word in $F_d$
such that the resulting 2 generator 1 relator group 
$H(w)=\langle x,t|w(x,t)\rangle$ is a semidirect product 
$F_r\rtimes_\alpha\z$ of a free group of rank $r$ (for $r\ge 2$) 
with the integers, and where
$t$ generates $\z$ with $x\in F_r$
and the automorphism $\alpha$ is conjugation by $t$.
If this is the case then let us call $w$ a {\bf free-by-cyclic} word
of rank $r$.
In fact it is straightforward and well known
to tell if a given $w$ is a free-by-cyclic
word: first suppose that the cyclically reduced word
$v(x_0,x_1,\ldots ,x_{d-1})$
is such that the generator with the smallest index that actually appears
in $v$ is $x_s$, and similarly $x_l$ is the largest, where
$0\le s\le l\le d-1$ and $r=l-s$.
The condition which must hold
is that $x_s$ appears only once in $v$, either as
itself or as its inverse, and the same for $x_l$.
If so then the associated polynomial
$f_w(t)$ is monic at both ends (both the largest non zero coefficient $a_l$
and smallest non zero coefficient $a_s$ are equal to $\pm 1$) and $r=l-s\le d$.
We then have:
\begin{thm}
If $v$ is any cyclically reduced word in $F_d$ giving rise to a
free-by-cyclic word $w$ of rank $r$ at least 2 then
any non abelian finite simple group appears as a quotient of the
cyclically presented group $G_n(v)$ for some $n$, indeed infinitely many $n$.
This is also true if given a finite list of non abelian finite simple groups:
there will be infinitely many $n$ such that $G_n(v)$ surjects to all of these
groups.
\end{thm}
\begin{proof}
This is an application of Corollaries 4.2 and 4.3 with $N$ equal to $F_r$ and
$G$ taken to be
$H(w)=\langle x,t\rangle=F_r\rtimes_\alpha\z$. 
We conclude that infinitely
many cyclic covers $\langle F_r,s=t^n\rangle$ of $H(w)$ surject to 
all of these finite 
simple groups. Now if we regard the $n$th cyclic cover as the kernel of the 
map $\epsilon_n(t):H(w)\rightarrow C_n$ given by the exponential sum of $t$
modulo $n$, this homomorphism will factor through $H_n(w)$ as we are just
adding the relator $t^n$. Now the kernel of this map from $H_n(w)$ to
$C_n$ is $G_n(v)$, thus on taking the $n$th cyclic cover and setting
$s$ to be the identity, the quotient so obtained is $G_n(v)$. But the
maps from our cyclic covers to finite simple groups given in
Proposition 4.1 involve setting $s=t^n$ equal to the identity,
so they factor through $G_n(v)$.
\end{proof}

Consequently in Theorem 5.2 we will have that infinitely many $G_n(v)$
are non soluble. However easy examples show that we need not have all
but finitely many $G_n(v)$ being non soluble.
\begin{ex}
Let $v\in F_4$ be $x_3x_0^{-1}$, so that $w(x,t)=t^3xt^{-3}x^{-1}$ which
is a free by cyclic word of rank 3.
Then 
\[G_n(v)=\langle x_0,x_1,\ldots ,x_{n-1},t|x_i=x_{i+3} (0\le i\le n-1)
\rangle\]
which is $F_3$ if 3 divides $n$ but $\z$ if not.
\end{ex}

We remark that free by cyclic words seem to be common but not generic:
in \cite{dlt} it was shown that the proportion of cyclically reduced words 
in $F_2$ of length $l$ that give a free by cyclic group has lim sup strictly
less than 1 but lim inf strictly bigger than 0 as $l$ tends to infinity.
However experimental evidence suggests a limit of about $0.94$. (Again
the point applies about zero exponent sum.)

We also mention that Thurston's theorem on orbifold Dehn filling can
sometimes be used in this context. If we have a free by cyclic group
$F_r\rtimes_\alpha\z=\langle F_r,t\rangle$ for $r\ge 2$
which is the fundamental group of a finite volume
orientable hyperbolic 3-manifold $M$ then $M$ is fibred over the circle
with fibre an orientable surface with boundary having fundamental  
group $F_r$. A necessary but not sufficient condition for this is that
there exists a non identity $x\in F_r$ such that some positive power of
$\alpha$ sends $x$ to a conjugate of itself.
Consequently $\pi_1(M)$ contains $\z\times\z$ and is not a word
hyperbolic group. Suppose we have $\alpha(x)=x$, which can be achieved by
first replacing $\alpha$ with a power $\beta=\alpha^k$, 
thus moving to a cyclic cover, and then
multiplying $\beta$ by an inner automorphism, preserving the group and
the manifold but changing the element $s=t^k$. Then $s$ commutes with $x$ and
so both are parabolic elements, corresponding to a torus boundary component.
On adding the relation $s^n$ to our group,
which corresponds
to gluing in a solid torus to this boundary component such that $n$ copies
of the loop on this component represented by $s$ are identified with the
compressible curve, we obtain an orbifold and Thurston's result tells us
that for all sufficiently large $n$ the result is a hyperbolic orbifold.
Thus we have that if $H(x,t)$ is a free by cyclic word of rank $r\ge 2$
that is the fundamental group of a finite volume
orientable hyperbolic 3-manifold and $t$ conjugates a non trivial element
of the free group to itself then $H_n(w)$ and $G_n(v)$ are infinite and
indeed contain non abelian free groups for all sufficiently large $n$. If
the manifold has only one cusp, as would be the case if the free group
was of rank 2, then the resulting hyperbolic orbifold is closed so that
addition of the relation $t^n$ has turned a non word hyperbolic group into    
a word hyperbolic group. However the original group is relatively hyperbolic
with respect to the boundary elements and the idea of adding a high
powered relation stems from Thurston's results in 3 dimensional hyperbolic
geometry.

Finally we finish this section with a famous example which shows that
looking at finite images cannot give the whole picture.
\begin{ex}
\end{ex}
Let $v\in F_2$ be $x_1x_0x_1^{-1}x_0^{-2}$, so that
$w=txt^{-1}xtx^{-1}t^{-1}x^{-2}$. The word $v$ really goes back to an
example of Higman, whereas $H(x,t)=\langle x,t|w(x,t)\rangle$ is known
as the Baumslag-Gersten group and it was shown in \cite{ba} that all
finite quotients of $H(x,t)$ are cyclic with $x$ mapping to the identity.
This must also be the case
for $H_n(w)=\langle x,t|w(x,t),t^n\rangle=G_n(v)\rtimes C_n$ and this
property remains true for finite index subgroups: suppose $G\le_f H$ and
all finite index subgroups of $H$ contain its commutator subgroup $H'$.
Then any finite index subgroup $L$ of $G$ is also one of $H$ and $G'\le H'$
implies that $G'\le L$. But $G_n(v)$ is a perfect group for all $n$ so has
no proper finite index subgroups at all. However Theorem 3 of \cite{prf} shows
that $G_n(v)$ and $H_n(w)$ are infinite for all $n\ge 4$.

\section{Ascending HNN extensions}

One generalisation of a semidirect product of the form $N\rtimes_\alpha\z$
is an ascending HNN extension $N*_\theta$. Whereas $\alpha$ must be an
automorphism of $N$, we only require that $\theta:N\rightarrow N$ is an
injective homomorphism, not necessarily surjective (although $N$ needs to
have proper subgroups isomorphic to itself, namely $N$ is non Hopfian, in
order for existence of a $\theta$ which is not an automorphism). If
$\langle X|R\rangle$ is a presentation for $N$ then, on taking a stable
letter $t$ we obtain the presentation $\langle X,t|R,txt^{-1}=\theta(x)
\forall x\in X\rangle$.

Ascending HNN extensions sometimes have comparable properties to semidirect
products, so we can ask whether $N$ having all finite groups as virtual
images implies the same for $N*_\theta$. In fact we can even ask the
same question for any HNN extension in which $N$ is the base. It is certainly
not true for non ascending HNN extensions, which is where we have an
isomorphism $\theta:A\rightarrow B$ of the associated subgroups $A$ and
$B$ of $N$, with both $A$ and $B$ proper subgroups.
To see this, let $N=\langle x,y|x^3=y^2\rangle$ which is the
fundamental group of the trefoil knot, thus has every finite group as a
virtual image (for a number of reasons, for instance by Corollary 3.2
as the knot is fibred). But on taking $A=\langle x\rangle$ and 
$B=\langle y\rangle$ with $\theta(x)=y$ we get
\[N*_\theta=\langle x,y,t|txt^{-1}=y,x^3=y^2\rangle\]
which on eliminating $y$ is seen to be the famous non Hopfian Baumslag
Solitar group $BS(2,3)$ whose only finite quotients are metabelian.

Unfortunately a similar phenomenon can happen for ascending HNN extensions,
as was demonstrated in \cite{spws}, where the ascending HNN extension
$\Gamma=G*_\theta$ of the Grigorchuk group $G$ formed by using the
Lysenok extension is shown to have all finite images metabelian: indeed
the image of $G$ is shown to be a quotient of $(C_2)^2$. As is
similar to the argument in Example 5.4, if a group $\Gamma$ has every finite
image metabelian then the finite residual $R_\Gamma$ (the intersection
of all finite index subgroups) would contain the second derived group
$\Gamma''$. But as $R_\Gamma=R_\Delta$ for any $\Delta\le_f\Gamma$, we
would have $\Delta''\le\Gamma''\le R_\Delta$ so all finite images of 
$\Delta$ are metabelian too. Now the Grigorchuk group $G$ is a 2-group
so certainly does not have every finite group as a virtual image: only
finite 2-groups, which must be nilpotent, can appear here. But $G$ has a
much wider range of finite images than the finite index subgroups of
$\Gamma$: if $G$ had only metabelian finite images then it would be
metabelian itself, as $G$ is residually finite so $R_G=I$. However this
is not true as $G$ is a finitely generated infinite torsion group. In
particular $\Gamma$ cannot be metabelian and so this paper gives us an
example of a non residually finite ascending HNN extension where the
base is finitely generated and residually finite. This is in contrast to
semidirect products where Malce'ev showed that if $N$ is finitely
generated then $N\rtimes H$ is residually finite if both $N$ and $H$ are
too. The proof is essentially Proposition 3.1, although we again remind
ourselves of the example in \cite{bminf} showing that this result fails
if $N$ is not finitely generated.      

This suggests that if we are given a finitely generated residually finite  
group $N$ which has every finite group as a virtual quotient then it seems
unreasonable to expect that an ascending HNN extension $N*_\theta$ will
have this property too unless we already know that the extension is
residually finite as well. One case where this has been established is in
\cite{brsp} which shows that ascending HNN extensions of free groups $F_r$
are residually finite. In contrast to the quick proof for semidirect 
products, this argument is deep and highly non trivial, involving material
in algebraic geometry (further use is made of this area, as well as some
model theory, to generalise the conclusion to ascending HNN extensions of
finitely generated linear groups). Whilst we do not invoke this theorem to 
establish that ascending HNN extensions of free groups $F_r$
have every finite
group as a virtual quotient (which we leave open), in the course of looking
for a proof we were able to come up with a considerable simplification of
the residually finite result for certain endomorphisms; those that induce
an injective map on the abelianisation of $F_r$.

To provide the necessary background, first note that 
any ascending HNN extension $\Gamma=G*_\theta$ with stable letter $t$ has
an associated homomorphism $\chi:\Gamma\rightarrow\z$ given by the
exponent sum of $t$ in an element of $\Gamma$. Now $\Gamma$ is also a
semidirect product $K\rtimes\z$ where $K=\mbox{ker}(\theta)$ but
$K=\cup_{i\in\n} t^{-i}Gt^i$ which is an ascending union, and a strictly
ascending union if $\theta$ is not surjective (which means that $K$ is
not finitely generated). Thus any element not in $K$ is preserved under some
homomorphism to a finite cyclic group. Moreover any element in $K$ is
conjugate to one in $G$, so if a conjugate of $g\in G$ is in $R_\Gamma$,
$g$ will be as well because $R_\Gamma\unlhd\Gamma$.

Consequently if $G$ is residually finite, in order
to establish residual finiteness for an ascending HNN extension
$\Gamma=G*_\theta$ we need only
consider the non identity elements $x$ of $G$ and look for some finite index 
subgroup $\Delta\le\Gamma$
with $x\notin\Delta$. We would like to use the fact that we have finite index
subgroups $H$ of $G$ with $x\notin H$, for instance we would be done if such
an $H$ somehow gave rise to a $\Delta$ satisfying $\Delta\cap G=H$.
However in the strictly ascending situation, there are severe restrictions
on which $H\le_f G$ are the intersection with $G$ of a finite index subgroup
of $\Gamma$.

\begin{prop} 
If $\Gamma=G*_\theta$ is an ascending HNN extension of a
group $G$ and $H\le_f G$ then there exists $\Delta\le_f\Gamma$
with $\Delta\cap G=H$ if and only if there is $l>0$ with $\theta^{-l}(H)=H$.
\end{prop}
\begin{proof}
Suppose on being given $H$ we have such a $\Delta$. Being of finite index
implies there is $l>0$ with $t^l\in\Delta$. As any $h\in H$ is in $\Delta$,
we have that $t^lht^{-l}=\theta^l(h)$ is in $\Delta$ but also in $G$, so
$\theta^l(H)\le H$, implying $H\le\theta^{-l}(H)$.
Now take $g\in\theta^{-l}(H)$, so that $g\in G$ of course. 
We have $g=t^{-l}h_0t^l$ for some $h_0\in H$, and $t^l,h_0\in\Delta$ implies
that $g$ is too, thus $g\in H$. 

Conversely it is shown in
\cite{bis} Proposition 4.3 (iv) by a short but careful argument that if
$H$ is any finite index subgroup of $G$ then $\langle H,t\rangle\le_f\Gamma$.
Now, just as for semidirect products over $\z$, we have cyclic covers
$\Gamma_n=\langle G,t^n\rangle$ of $\Gamma$ which are themselves 
ascending HNN extensions with $s=t^n$ as stable letter, formed by using
the endomorphism $\theta^n$. Thus if we have $H=\theta^{-l}(H)=
\theta^{-2l}(H)=\theta^{-3l}(H)=\ldots$ then $\Delta=\langle H,s=t^l\rangle
\le_f\Gamma_l=\langle G,s\rangle\le_f\Gamma$. It is clear that 
$H\le\Delta\cap G$ so let $g\in\Delta\cap G$. As $\theta^l(H)\le H$, we have
that $\Delta$ is also an ascending HNN extension with stable letter $s$, by
restricting $\theta^l$ to $H$. This means that any element of $\Delta$ can
be expressed in the form $s^{-p}hs^q$ for $p,q\ge 0$ and $h\in H$. Now if
$g=s^{-p}hs^q$ then we must have $p=q$ as $g$ is in the kernel of the
associated homomorphism (which is just restriction to $\Delta$ of that for
$\Gamma$). Thus $\theta^{pl}(g)=s^pgs^{-p}\in H$, meaning that 
$g\in\theta^{-pl}(H)=H$.
\end{proof}

As for finding such subgroups which are invariant under pullback by (a power
of) $\theta$, it is shown in \cite{bis} Theorem 4.4 that if $G$ is
finitely generated (which henceforth we will assume) then on
repeatedly pulling back $H$ via $\theta$, we obtain $\theta^{-k}(H)=
\theta^{-k-l}(H)$ for some $k\ge 0$ and $l>0$. This means that on setting
$L=\theta^{-k}(H)$ we have $\theta^{-l}(L)=L$. However it could well be
that we find $L$ is all of $G$ anyway. What is required is a good supply
of fully invariant subgroups, meaning that $\theta(L)\le L$ for any 
endomorphism $\theta$, which implies that $L$ is contained in 
$\theta^{-1}(L)$.

Now further suppose that $L$ has finite index in $G$. In this case we 
would have $L=\theta^{-1}(L)$ if and only if $[G:L]=[G:\theta^{-1}(L)]$.
In fact this happens if and only if $\theta(G)L=G$. This follows because
the right hand side is equal to
$[\theta^{-1}\theta(G):
\theta^{-1}(L\cap\theta(G)]$, and as $\theta(G)$ and 
$L\cap\theta(G)$ are obviously
in the image of $\theta$, this index is preserved on
removing $\theta^{-1}$ to get $[\theta(G):L\cap\theta(G)]=[\theta(G)L:L]$.

Possibilities for these fully invariant subgroups are, given a prime $p$,
the derived $p$-series and the lower central $p$-series, both of which
intersect in the identity in the case of a free group
$F_r$ and have first term $F_r^p[F_r,F_r]$
with quotient $(C_p)^r$. 

\begin{thm} If $F_r$ is the free group of rank $r\ge 2$ and $\theta$ is an
injective endomorphism of $F_r$ then consider the induced homomorphism
of abelianisations
$\overline{\theta}:\z^r\rightarrow\z^r$ given by $\theta(x)[F_r,F_r]=
\overline{\theta}(x[F_r,F_r])$. If det$(\overline{\theta})\ne 0$ then the
ascending HNN extension $F_r*_\theta$ is residually finite.
\end{thm}
\begin{proof}
Given any prime $p$, we can consider the endomorphism $\overline{\theta}_p$
of $(C_p)^r$ by reducing $\overline{\theta}$ 
mod $p$. If det$(\overline{\theta})\ne 0$
when considered as an endomorphism of $\z^r$ then, by taking a prime $p$
which does not divide det$(\overline{\theta})$ we have that 
$\overline{\theta}_p$ is invertible.

Thus in the case where $G$ is the free group $F_r$,
on being given a non identity element $x$ of $G$ we choose a term
$L_i$ of the derived or other appropriate $p$-series for $F_r$ where
$x\notin L_i$. Then $\theta(L_i)\le L_i\unlhd_fG$, allowing us to take the
finite index subgroup $S_i=\langle L_i,t\rangle$ of $\Gamma$. We are
done if we can show $\theta(G)L_i=G$ because then we would have 
$\theta^{-1}(L_i)=L_i$ by the above, so we can apply Proposition 6.1 to
conclude that $S_i\cap G=L_i$ and $x\notin S_i$. Now $\theta(G)L_i=G$
if and only if $\theta(G)L_i/L_i=G/L_i$ but $\theta(G)L_i/L_i$
is the image of
$\theta(G)$ under the quotient map $q_i$ from $G$ to $G/L_i$.

We certainly have that $q_i(\theta(G))=G/L_i$ if $i=1$ in which case
$L_1=G^p[G,G]$, since by our assumption on $p$ we have that $q_1\theta(G)
=\overline{\theta}_p((C_p)^r)$
is all of $(C_p)^r=G/L_1$. However as $G/L_i$ is a finite $p$-group, we
can utilise the Frattini subgroup (intersection of all maximal subgroups).
If this subgroup is finitely generated then a set generates the
whole group
if and only if it generates the group when quotiented by the Frattini
subgroup.  Now in the case of a finite $p$-group $P$, the Frattini
subgroup is $P^p[P,P]$. Thus for any $i$ we have that the Frattini
subgroup of $P=G/L_i$ is the image of $G^p[G,G]$ under $q_i$, so the
quotient of $G/L_i$ by this subgroup is $G/(G^p[G,G]L_i)$. But
$L_i\le G^p[G,G]=L_1$ and so the image of $\theta(G)$ in $G/L_i$ is
all of $G/L_i$.
\end{proof}

We have written out this proof so that it applies in more general
situations:
\begin{co}
If $G$ is any finitely generated group which is residually finite $p$
and $\theta$ is an injective endomorphism of $G$ then the
associated HNN extension $G*_\theta$ is residually finite provided that
the induced endomorphism of $G/G^p[G,G]$ is invertible.
\end{co}
\begin{proof}
The residually finite $p$ condition is equivalent to the derived or
lower central $p$-series intersecting in the identity. Now the proof
proceeds as before and the invertible condition is used to invoke the
Frattini argument at the end.
\end{proof}

There are a range of groups which are residually finite $p$, for instance
any finitely generated linear group in characteristic 0 is virtually
residually finite $p$ for all but finitely many primes $p$ (again due
to Malce'ev), thus we can find examples amongst the finite index subgroups
of any such linear group. We remark though that a necessary condition
for a non-cyclic
finitely generated group $G$ to be residually finite $p$ is that
$G$ surjects to $C_p\times C_p$, as otherwise all finite $p$-images of
$G$ are cyclic which would imply in this case that $G$ was too.\\
\hfill\\
Example: The finitely presented
ascending HNN extension $\Gamma=G*_\theta$ of the Grigorchuk group
$G=\langle a,c,d\rangle$ which is shown to be non residually finite
in \cite{spws} is formed using the injective endomorphism $\sigma$,
where $\sigma(a)=aca,\sigma(c)=dc,\sigma(d)=c$. As $G/G^2[G,G]=(C_2)^3$,
generated by the images of $a,c,d$, we see that the induced homomorphism
on $G/G^2[G,G]$ is not injective, with $ad$ in the kernel. (If it were
then Corollary 6.3 would give us the first example of a finitely presented,
residually finite group which is not virtually soluble nor contains $F_2$.
This is because $G$ is a $2$-group and residually finite, so residually
finite 2.)

In fact we can adjust $\Gamma$ slightly to come up with a group which is
``even less residually finite''. Any ascending
HNN extension must surject to $\z$
and so have some finite index subgroups, namely the cyclic covers. However
here we have an example where these are all the finite index subgroups, even
though the base is finitely generated and residually finite.
\begin{prop}
Let $G$ be the Grigorchuk group and $[G,G]$ its commutator subgroup
of index 8. Then the only finite index subgroups of 
the ascending HNN extension $\Delta=[G,G]*_\sigma$, where
we restrict $\sigma$ to $[G,G]$, are the cyclic covers.
\end{prop}
\begin{proof}
We have that $\Delta$ is a finite index subgroup of $\Gamma=G*_\sigma$
by \cite{bis} Proposition 4.3 (iv); in fact it can be checked that
$\Delta$ has index 4. Moreover it is shown in \cite{spws} that
in any finite quotient of $\Gamma$, the base $G$ maps to an abelian subgroup
and so $[G,G]$ maps to the identity. Suppose there exists a finite index
normal subgroup $N$ of $\Delta$ such that the image of $[G,G]$ is non trivial
in $\Delta/N$. Although $N$ need not be normal in $\Gamma$, we can find
$M\unlhd_f\Gamma$ with $M\le N$, so that the image of $[G,G]$ is non trivial
in $\Gamma/M$ which is a contradiction.

This means that every finite index normal subgroup of $\Delta$ contains 
$[G,G]$ and hence also the kernel of the associated homomorphism for
$\Delta=[G,G]*_\sigma$, thus the only finite quotients of $\Delta$ are
cyclic and the only finite index subgroups are the cyclic covers.
\end{proof}
Example: In \cite{spdt} the injective endomorphism $\theta(a)=b,
\theta(b)=a^2$ of the rank two free group $F(a,b)$ is considered. The
resulting ascending HNN extension $F(a,b)*_\theta$ is shown to be a
1-relator group $\langle a,t|t^2at^{-2}=a^2\rangle$ which is non linear
(by using results of Wehrfritz) but residually finite (using \cite{brsp}).
Here we see the conditions in Theorem 6.2 are satisfied because
det$(\overline{\theta})=-2$, 
so we can use any prime but 2 to complete an
proof that $F(a,b)*_\theta$ has these properties without recourse to the
sophisticated results in \cite{brsp}.

\end{document}